\documentclass[reqno]{amsart}
\usepackage{amssymb}
\usepackage{hyperref}

\begin{document}
\title[ ]
{Asymmetric Fuglede Putnam's Theorem for operators reduced by their eigenspaces}

\author[ F. LOMBARKIA and M. AMOUCH]
{  F. Lombarkia and M. Amouch }  

\address{Farida Lombarkia \newline
 Department of Mathematics, Faculty of Science, University of Batna,
05000, Batna, Algeria.}
\email{lombarkiafarida@yahoo.fr}

\address{Mohamed AMOUCH \newline
Department of Mathematics and informatic
University Chouaib Doukkali,
Faculty of Sciences, Eljadida.
24000, Eljadida, Morocco.}
\email{mohamed.amouch@gmail.com}




\begin{abstract}
Fuglede-Putnam Theorem have been proved for a considerably large number of class of operators. In this paper by using the spectral theory, we obtain a theoretical and general framework from which Fuglede-Putnam theorem may be promptly established for many classes of operators.
\end{abstract}

\maketitle
\numberwithin{equation}{section}
\newtheorem{theorem}{Theorem}[section]
\newtheorem{lemma}[theorem]{Lemma}
\newtheorem{proposition}[theorem]{Proposition}
\newtheorem{corollary}[theorem]{Corollary}
\newtheorem*{remark}{Remark}
\newtheorem{defn}[theorem]{Definition}
\newtheorem{exa}[theorem]{Example}

\section{Introduction and basic definitions}
Fuglede-Putnam Theorem have been studied in the last two decades by several authors and most of them have essentially proved such theorem for special classes of operators. Many times the arguments used, to prove Fuglede-Putnam Theorem are similar, but in this paper we show that it is possible to bring back up this theorem from some general common properties. We use the spectral theory to obtain a theoretical and general framework from which Fuglede-Putnam theorem is established, and we can deduce that Fuglede-Putnam Theorem hold for many classes of operators. Let $H$ be an infinite complex Hilbert space and consider two bounded linear operators
$A,B\in L(H)$. Let $L_{A}\in L(L(H))$ and $R_{B}\in L(L(H))$ be the left and the right
multiplication operators, respectively, and denote by $d_{A,B}\in L(L(H))$ either the
elementary operator $\Delta_{A,B}(X)=AXB-X$ or the generalized derivation
$\delta_{A,B}(X)=AX-XB$. Given $T\in L(H),$ $\ker(T)$, $\mathcal{R}(T)$, $\sigma(T)$ and $\sigma_{p}(T)$ will stand for the null space, the range of $T$, the spectrum of $T$ and the point spectrum of $T$. Recall that if $M,$ $N$ are linear subspaces of a normed linear space $V,$ then
$M$ is orthogonal to $N$ in the sense of Birkhoff, $M\bot N$ for short, if $\|m\|\leq\|m+n\|$ for all $m\in M$ and $n\in N.$

It is known that if $A,\,B^{*}\in L(H)$ are hyponormal operators, then $d_{A,B}$ satisfies the asymmetric Putnam Fuglede commutativity property $\ker d_{A,B}\subseteq\ker d_{A^{*},B^{*}},$ hence $\ker d_{A,B}\bot\mathcal{R}(d_{A,B}).$
From the fact that hyponormal operators are closed under translation and multiplication by scalars, B. P. Duggal in \cite{Du8} deduced that
if $A$ and $B^{*}$ are hyponormal, then
$\ker(d_{A,B}-\lambda I)\subseteq\ker(d_{A^{*},B^{*}}-\overline{\lambda} I)$ holds for every complex number
$\lambda,$ where $\overline{\lambda}$ is the conjugate of $\lambda.$
An operator $T\in L(H)$ is said to be p-hyponormal, $0<p\leq 1,$ if $|T^{*}|^{2p}\leq|T|^{2p},$
where $|T|= (T^*T)^{1\over 2}$. An invertible operator $T\in L(H)$ is log-hyponormal if $\log|T^{*}|^{2}\leq\log|T|^{2}.$
In \cite{Du8} B. P. Duggal proved that if $A$ and $B^{*}$ are p-hyponormal or log-hyponormal, then $\ker(d_{A,B}-\lambda I)\subseteq\ker(d_{A^{*},B^{*}}-\overline{\lambda} I)$ and $\ker (d_{A,B}-\lambda I)\bot\mathcal{R}(d_{A,B}-\lambda I),$ for every complex number $\lambda.$
An operator $T\in L(H)$ is said to be w-hyponormal if $( {| T^{*}|}
^{\frac{1}{2}}{| T|} {| T^{*}|} ^{\frac{1}{2}}) ^{\frac{1}{2}}\geq{|T^{* }|},$ see \cite{Itya}.
It is shown in \cite{AW,AW2} that the class of w-hyponormal properly contains the class of p-hyponormal ($0<p\leq 1,$) and log-hyponormal.
T. Furuta, M. Ito and T. Yamazaki \cite{FIY} introduced a very interesting class $\mathcal{A}$ operators defined by $|T^{^2}|-|T|^{2}\geq0,$ and they showed that class $\mathcal{A}$ is a subclass of paranormal operators (i.e., $\|Tx\|^{2}\leq\|T^{2}x\|\|x\|,$ for all $x\in H$ ) and contains w-hyponormal operators. An operator $T\in L(H)$ is said to be class $\mathcal{A}(s,t)$, $0<s,t$ if $|T^{*}|^{2t}\leq(|T^{*}|^{t}|T|^{2s}|T^{*}|^{t})^{\frac{t}{t+s}}.$ Then $T\in\mathcal{A}(\frac{1}{2},\frac{1}{2})$ if an only if $T$ is w-hyponormal  and $T\in\mathcal{A}(1,1)$ if an only if $T$ is class $\mathcal{A}$. \\I. H. Jeon and I. H. Kim \cite{JK} introduced quasi-class $\mathcal{A}$ operators defined by $T^{*}(|T^{^2}|-|T|^{2})T\geq0,$ as an extension of the notion of class $\mathcal{A}$ operators. K. Tanahash, I. H. Jeon, I, H. Kim and A. Uchiyama \cite{TJKU} introduced k-quasi-class $\mathcal{A}$ operators defined by $T^{*k}(|T^{^2}|-|T|^{2})T^{k}\geq0,$ for a positive integer $k$ as an extension of the notion of quasi-class $\mathcal{A}$ operators, for interesting properties of k-quasi-class $\mathcal{A}$ operators, called also quasi-class ($\mathcal{A}$, k), see \cite{GF, TJKU}.\\
In \cite[Lemma 2.4]{CH}, \cite[Theorem 3.6]{BL} and \cite[Lemma 2.4]{DKK} the authors proved that if $A,\,B^{*}\in L(H)$ are w-hyponormal operators with
$\ker A\subseteq\ker A^{*}$ and $\ker B^{*}\subseteq\ker B$, then $d_{A,B}$ satisfies the asymmetric Putnam Fuglede commutativity property $\ker d_{A,B}\subseteq\ker d_{A^{*},B^{*}}.$
Recently B.P. Duggal, C. S. Kubruslly and I. H. Kim in \cite[Theorem 2.5]{DKK} have proved that if $A\in\mathcal{A}(s_1,t_1)$  and $B^{*}\in\mathcal{A}(s_2,t_2)$, $0<s_1,s_2,t_1,t_2\leq 1$ are such that
$\ker A\subseteq\ker A^{*}$ and $\ker B^{*}\subseteq\ker B$, then $\delta_{A,B}$ satisfies the asymmetric Putnam Fuglede commutativity property $\ker \delta_{A,B}\subseteq\ker \delta_{A^{*},B^{*}}.$ \\In this paper we prove that all the precedent results are a consequence of our main results. Since the class of  k-quasi-class $\mathcal{A}$ operators contains properly the class $\mathcal{A}(s,t)$, $0<s,t$
operators, it therefore provides a unified approach in studying the asymmetric Putnam Fuglede commutativity property for k-quasi-class $\mathcal{A}$ operators.\\ Now we recall some definitions

\begin{defn}
An operator $T\in L(H)$ has Bishop's property $(\beta)$ if for every open set $U\subset\mathbb{C}$ and every sequence of analytic functions $f_{n}:U\rightarrow X,$ with the property that $(T-\lambda I)f_{n}(\lambda)\rightarrow0$ uniformly on every compact subset of $U,$ it follows that $f_{n}\rightarrow0,$ again locally uniformly on $U.$
\end{defn}
Bishop's property $(\beta)$ implies Dunford property $(C)$, also $T$ satisfies property $(\beta)$ if and only if $T^{*}$ satisfies property $(\delta)$ \cite[Theorem 2.5.5]{LN}. For more information on  property $(\beta)$, property $(\delta)$ and Dunford's condition $(C)$ we refer the interested reader to \cite{LN}.

Recall that the ascent $p(T)$ of an operator $T$, is defined by
$p(T)=\inf\{n\in\mathbb{N}:\ker T^{n}=\ker T^{n+1}\}$ and the descent
$q(T)=\inf\{n\in\mathbb{N}:R(T^{n})=R(T^{n+1})\}$, with
$\inf\emptyset=\infty.$ It is well known that if $p(T)$ and $q(T)$ are both finite
then $p(T)=q(T).$  We denote by
$\Pi(T)=\{\lambda\in\mathbb{C}: p(T-\lambda I)=q(T-\lambda I)<\infty\}$ the set of poles of the resolvent. An operator $T\in L(H)$ is called Drazin invertible if and only if it has finite ascent and descent. The Drazin spectrum of an operator $T$ is defined by
\begin{equation*}
\sigma_{D}(T)=\{\lambda\in\mathbb{C}: T-\lambda I\,\,\,\text{%
is not Drazin invertible}\}.
\end{equation*}
In the sequel we shall denote by $accD$ and $isoD,$ the set of accumulation points and the set of isolated points of $D\subset\mathbb{C},$ respectively
\begin{defn}
An operator $T\in L(H)$ is said to be polaroid if
\begin{equation*}
iso\sigma(T)\subseteq\Pi(T).
\end{equation*}
\end{defn}
It is easily seen that, if $T\in L(H)$ is polaroid, then $\Pi(T)=E(T)$, where $E(T)$ is the set of eigenvalues of $T$ which are isolated in the spectrum of $T$\\
An important subspace in local spectral theory is the the quasinilpotent part of $T$ is defined by
\begin{equation*}
H_{0}(T)=\{x\in H: \lim_{n\rightarrow\infty}\|T^{n}(X)\|^{\frac{1}{n}}=0\}.
\end{equation*}
It is easily seen that $\ker T^{n}\subset H_{0}(T)$ for every $n\in\mathbb{N}$, see \cite{A} for information on $H_{0}(T)$.\\

The range-kernel orthogonality of $d_{A,B}$ in the sense of G. Birkhoff was studied by numerous mathematicians, see \cite{Am,BL, D, K, Tu} and the references therein. A sufficient condition guaranteeing the range-kernel orthogonality of $d_{A,B}$ is that $\ker d_{A,B}\subseteq\ker d_{A^{*},B^{*}}$ \cite{D}.
The main objective of this paper is to give sufficient conditions to have $\ker(d_{A,B}-\lambda I)\subseteq \ker(d_{A^{*},B^{*}}-\overline{\lambda} I),$ for every complex number $\lambda$. After section one where several basic definitions and facts will be recalled, in section two, we prove that if $A$ and $B^{*}$ are reduced by each of its eigenspaces, polaroid and have property $(\beta),$ then $\ker(d_{A,B}-\lambda I)\subseteq \ker(d_{A^{*},B^{*}}-\overline{\lambda} I)$ for every complex number $\lambda$ and that the elementary operator $d_{A,B}$ satisfies the range-kernel orthogonality $\|X\|\leq\|X-(d_{A,B}-\lambda I)Y\|$, for all $X\in\ker(d_{A,B}-\lambda I)$ and $Y\in L(H).$
We apply the results obtained to k-quasi-class $\mathcal{A }$ operators. We prove that $d_{A,B}$ satisfies the asymmetric Putnam Fuglede commutativity property $\ker d_{A,B}\subseteq\ker d_{A^{*},B^{*}},$ if $A$ and $B^{*}$ are k-quasi-class $\mathcal{A}$ operator with $\ker A\subseteq\ker A^{*}$ and $\ker B^{*}\subseteq\ker B.$ This generalizes results given in \cite[Theorem 2.3]{Du8}, \cite[Theorem 3.6]{BL}, \cite[Lemma 2.4]{DKK} and \cite[Theorem 2.5]{DKK}.

\section{Main results}

Let $T \in L(H)$ be reduced by each of its eigenspaces. If we let $M=\bigvee\{\ker (T-\mu I), \mu\in\sigma_{p}(T)\}$ (where $\bigvee(.)$ denotes the closed linear span, it follows that $M$ reduces $T$. Let $T_{1}=T|_{M}$ and $T_{2}=T|_{M^{\perp}}$. By \cite[Proposition 4.1]{Ber} we have
\begin{itemize}
  \item $T_{1}$ is normal with pure point spectrum,
  \item $\sigma_{p}(T_{1})=\sigma_{p}(T)$,
  \item $\sigma(T_{1})=cl\sigma_{p}(T_{1})$(here $cl$ denotes closure),
  \item $\sigma_{p}(T_{2})=\emptyset$.
\end{itemize}

The classical and most known form of the Fuglede Putnam theorem is the following
\begin{theorem}\cite{F, P}
If $X, A$ and $B$ are bounded operators acting on complex Hilbert space $H$ such that $A$ and $B$ are normal, then
\begin{equation*}
AX=XB\Longrightarrow A^{*}X=XB^{*}.
\end{equation*}
\end{theorem}

\begin{theorem}\label{theorem1}
Suppose that $A,B^{*}\in L(H)$ are reduced by each of its eigenspaces, polaroid and have property $(\beta)$, then
\begin{equation*}
\ker(\delta_{A,B}-\lambda I)\subseteq\ker(\delta_{A^{*},B^{*}}-\overline{\lambda} I), \,\,\ \forall\lambda\in\mathbb{C}.
\end{equation*}
\end{theorem}
\begin{proof}
Since $A$ and $B^{*}$ are reduced by each of its eigenspaces, then
 then there exists
 $$M_{1}=\bigvee\{\ker (A-\beta I), \beta\in\sigma_{p}(A)\} \mbox{ and } M_{2}=H\ominus M_{1}$$ on the one hand and $$N_{1}=\bigvee\{\ker (B^{*}-\overline{\alpha}I), \overline{\alpha}\in\sigma_{p}(B^{*})\} \mbox{ and } N_{2}=H\ominus N_{1}$$ on the other hand such that
$A$ and $B$ have the representations $$A=\left(
                                  \begin{array}{cc}
                                     A_1 & 0 \\
                                    0 & A_{2}\\
                                  \end{array}
                                \right)
\mbox{ on } H=M_{1}\oplus M_{2}$$  and $$ B=\left(
                                  \begin{array}{cc}
                                     B_1 & 0 \\
                                    0 & B_{2}\\
                                  \end{array}
                                \right)
\mbox{ on }H=N_{1}\oplus N_{2}$$ Recall from \cite{ER} that $\sigma(\delta_{A,B})=\sigma(A)-\sigma(B)$, we consider the following cases:

\textbf{Case 1:} If $\lambda\in\mathbb{C}\backslash\sigma(\delta_{A,B}),$ then $\ker(\delta_{A,B}-\lambda I)=\{0\}$ and hence $$\ker(\delta_{A,B}-\lambda I) \subseteq\ker(\delta_{A^{*},B^{*}}-\overline{\lambda} I).$$

\textbf{Case 2:} If $\lambda\in iso\sigma(\delta_{A,B}),$ then there exists finite sequences
 $\{\mu_{i}\}_{i=1}^{n}$ and $\{\nu_{i}\}_{i=1}^{n}$, where
 $\mu_{i}\in iso\sigma(A)$ and $\nu_{i}\in iso\sigma(B)$ such that $$\lambda=\mu_{i}-\nu_{i}, \mbox{ for all }1\leq i\leq n.$$
 $\lambda\not\in\sigma(\delta_{A_{i},B_{j}})$ for all $1\leq i,j\leq 2$ other than $i=j=1$. Consider $X\in\ker(\delta_{A,B}-\lambda I)$ such that $X:N_{1}\oplus N_{2}\longrightarrow M_{1}\oplus M_{2} \mbox{ have the representation } X=[X_{kl}]_{k,l=1}^{2}.$
Hence $$(\delta_{A,B}-\lambda I)(X)=\left(
                                  \begin{array}{cc}
                                    (\delta_{A_{1},B_{1}}-\lambda I)(X_{11}) & (\delta_{A_{1},B_{2}}-\lambda I)(X_{12}) \\
                                    (\delta_{A_{2},B_{1}}-\lambda I)(X_{21}) & (\delta_{A_{2},B_{2}}-\lambda I)(X_{22})\\
                                  \end{array}
                                \right)=0
.$$ Observe that $\delta_{A_{i},B_{j}}-\lambda I$ is invertible for all $1\leq i,j\leq2$ other than $i=j=1$. Hence $X_{22}=X_{21}=X_{12}=0.$ Since $A_{1}-\mu_{i}$ and $B_{1}-\nu_{i}$ are normal, it follows from Fuglede-Putnam theorem that $$(A_{1}^{*}-\overline{\mu_{i}})X_{11}-X_{11}(B_{1}^{*}-\overline{\nu_{i}})=0,$$ consequently $$X=X_{11}\oplus 0\in\ker(\delta_{A^{*},B^{*}}-\overline{\lambda} I).$$

\textbf{Case 3:} If $\lambda\in acc\sigma(\delta_{A,B}),$ then there exists
 $\mu\in\sigma(A)$ and $\nu\in\sigma(B)$ such that $\lambda=\mu-\nu\in(\sigma(A)-acc\sigma(B))$ or $\lambda=\mu-\nu\in(acc\sigma(A)-\sigma(B))$. Since $A$ and $B$ are polaroid, then $\mu\in acc\sigma(A)=\sigma_{D}(A)$ and $\nu\in acc\sigma(B)=\sigma_{D}(B)$, it follows from \cite[Theorem 2.4]{BZ} that
 $$\mu\in\sigma_{D}(A_{1})\cup\sigma_{D}(A_{2})\mbox{ and }\nu\in\sigma_{D}(B_{1})\cup\sigma_{D}(B_{2}),$$ since $\sigma_{p}(A_{2})=\sigma_{p}(B^{*}_{2})=\emptyset,$ then $\sigma_{D}(A_{2})=\sigma(A_{2})$ and $\sigma_{D}(B_{2})=\sigma(B_{2})$. Hence
 $$\mu\in\sigma_{D}(A_{1})\cup\sigma(A_{2})\mbox{ and }\nu=\sigma_{D}(B_{1})\cup\sigma(B_{2}).$$ Let $X\in\ker(\delta_{A,B}-\lambda I)$ such that $$X:N_{1}\oplus N_{2}\longrightarrow M_{1}\oplus M_{2} \mbox{ have the representation } X=[X_{kl}]_{k,l=1}^{2}.$$
Hence $$(\delta_{A,B}-\lambda I)(X)=\left(
                                  \begin{array}{cc}
                                    (\delta_{A_{1},B_{1}}-\lambda I)(X_{11}) & (\delta_{A_{1},B_{2}}-\lambda I)(X_{12}) \\
                                    (\delta_{A_{2},B_{1}}-\lambda I)(X_{21}) & (\delta_{A_{2},B_{2}}-\lambda I)(X_{22})\\
                                  \end{array}
                                \right)=0
.$$
We consider the following cases
\begin{itemize}
  \item $\mu\in\sigma(A_1)$ and $\nu\in\sigma_{D}(B_1)$, or
  \item $\mu\in\sigma(A_1)$ and $\nu\in\sigma(B_2)$, or
  \item $\mu\in\sigma(A_2)$ and $\nu\in\sigma_{D}(B_1)$, or
  \item $\mu\in\sigma(A_2)$ and $\nu\in\sigma(B_2)$.
\end{itemize}
or
\begin{itemize}
  \item $\mu\in\sigma_{D}(A_1)$ and $\nu\in\sigma(B_1)$, or
  \item $\mu\in\sigma_{D}(A_1)$ and $\nu\in\sigma(B_2)$, or
  \item $\mu\in\sigma(A_2)$ and $\nu\in\sigma(B_1)$.
\end{itemize}
We start by studying these cases
\begin{itemize}
  \item If $\mu\in\sigma(A_1)$ and $\nu\in\sigma_{D}(B_1)$.
  Observe that $\delta_{A_{i},B_{j}}-\lambda I$ is invertible for all $1\leq i,j\leq2$ other than $i=j=1$, we have $X\in\ker(\delta_{A,B}-\lambda I)$. Hence $X_{12}=X_{21}=X_{22}=0.$ Since $A_{1}-\mu$ and $B_{1}-\nu$ are normal, it follows from Fuglede-Putnam theorem that $$(A_{1}^{*}-\overline{\mu_{i}})X_{11}-X_{11}(B_{1}^{*}-\overline{\nu_{i}})=0,$$ consequently $$X=X_{11}\oplus 0\in\ker(\delta_{A^{*},B^{*}}-\overline{\lambda} I).$$

  \item If $\mu\in\sigma(A_1)$ and $\nu\in\sigma(B_2)$, we have $X\in\ker(\delta_{A,B}-\lambda I)$, then $X_{22}=X_{21}=X_{11}=0$.
  Since $\mu\in\sigma(A_1)=cl\sigma_{p}(T_{1})$, then there exist a sequence $(\mu_{n})\in\sigma_{p}(A_{1})$ such that $\mu_{n}\longrightarrow\mu$, so for all $n\in\mathbb{N}$ there exist a non null vector $x_{n}\in\ker(A_1-\mu_{n})=\ker(A^{*}_1-\overline{\mu_{n}})$. We have
  $(A_1-\mu)X_{12}=X_{12}(B_2-\nu)$ this implies that $(B_{2}^{*}-\overline{\nu})X_{12}^{*}=X_{12}^{*}(A_{1}^{*}-\overline{\mu})$, consequently $(B_{2}^{*}-\overline{\nu})X_{12}^{*} x_{n}=X_{12}^{*}(A_{1}^{*}-\overline{\mu})x_{n}.$ Since $\mu_{n}$ is an eigenvalue, so there exists a no null vector $z\in\ker(A^{*}_1-\overline{\mu_{n}})$, such that $\lim\|x_{n}-z\|=0.$ Since $(B_{2}^{*}-\overline{\nu})$ is injective, we obtain $X_{12}=0,$ hence $$X=0\in\ker(\delta_{A^{*},B^{*}}-\overline{\lambda} I).$$
  \item If $\mu\in\sigma(A_2)$ and $\nu\in\sigma_{D}(B_1)$, we have $X\in\ker(\delta_{A,B}-\lambda I)$, then $X_{11}=X_{22}=X_{12}=0$ and
  $(A_2-\mu)X_{21}=X_{21}(B_1-\nu)$. Since $\nu\in\sigma_{D}(B_1)$, then $\overline{\nu}\in\sigma_{D}(B^{*}_1)=\sigma(B^{*}_1)\backslash E(B^{*}_{1})$, so there exist a sequence $(\nu_{n})\in\sigma_{p}(B^{*}_{1})$ such that $\nu_{n}\longrightarrow\overline{\nu}$, so for all $n\in\mathbb{N}$ there exist a non null vector $x_n\in\ker(B^{*}_1-\nu_{n})=\ker(B_1-\overline{\nu_{n}})$. We have
  $(A_2-\mu)X_{21}=X_{21}(B_1-\nu)$, consequently $(A_2-\mu)X_{21}x_n=X_{21}(B_1-\nu)x_n.$ Since $\nu_{n}$ is an eigenvalue, so there exists a no null vector $z\in\ker(B_1-\overline{\nu_{n}})$, such that $\lim\|x_{n}-z\|=0.$ Since $(A_2-\mu)$ is injective, we obtain $X_{21}=0,$ hence $$X=0\in\ker(\delta_{A^{*},B^{*}}-\overline{\lambda} I).$$
  \item If $\mu\in\sigma(A_2)$ and $\nu\in\sigma(B_2)$. Since $A$ has property $(\beta)$ it follows from
  \cite[Remarks 3.2]{BZZ} that $A_{2}$ has property $(\beta)$, applying \cite[Theorem 2.20]{A} we get $H_{0}(A_2-\mu)$ is closed and from \cite[Proposition 1.2.20]{LN} That $\sigma(A_2|_{H_{0}(A_2-\mu)})\subseteq\{\mu\}$. If $\sigma(A_2|_{H_{0}(A_2-\mu)})=\emptyset$, then $H_{0}(A_2-\mu)=\{0\}$, the case $\sigma(A_2|_{H_{0}(A_2-\mu)})=\{\mu\}$ is not possible, since the operator $A_2$ does not contains isolated points. Hence $H_{0}(A_2-\mu)=\{0\},$ we have $X\in\ker(\delta_{A,B}-\lambda I)$, then $X_{21}=X_{12}=X_{11}=0$ and $(A_2-\mu)X_{22}=X_{22}(B_2-\nu)$, this implies that, if $t\in H_{0}(B_2-\nu)$, then $X_{22}t\in H_{0}(A_2-\mu)=\{0\}.$ Hence $X_{22}t=0.$ Since $t\in H_{0}(B_2-\nu)$, using properties of quasinilpotent part, we get $(B_2-\nu)(t)\in H_{0}(B_2-\nu),$ consequently $N_{2}=\overline{H_{0}(B_2-\nu)}.$ So $X_{22}=0$, hence $$X= 0\in\ker(\delta_{A^{*},B^{*}}-\overline{\lambda} I).$$
\end{itemize}
The other cases will be proved similarly.
\end{proof}
\begin{theorem}\label{theorem2}
Let $A,B\in L(H)$. If all the eigenvalues of $A$, $B^{*}$ are reduced by each of its eigenspaces, polaroid and have property $(\beta)$, then
\begin{equation*}
\ker(\Delta_{A,B}-\lambda I)\subseteq\ker(\Delta_{A^{*},B^{*}}-\overline{\lambda} I), \,\,\ \forall\lambda\in\mathbb{C}.
\end{equation*}
\end{theorem}

\begin{proof}
Since $A$ and $B^{*}$ are reduced by each of its eigenspaces, then
 then there exists
 $$M_{1}=\bigvee\{\ker (A-\beta I), \beta\in\sigma_{p}(A)\} \mbox{ and } M_{2}=H\ominus M_{1}$$ on the one hand and $$N_{1}=\bigvee\{\ker (B^{*}-\overline{\alpha}I), \overline{\alpha}\in\sigma_{p}(B^{*})\} \mbox{ and } N_{2}=H\ominus N_{1}$$ on the other hand such that
$A$ and $B$ have the representations $$A=\left(
                                  \begin{array}{cc}
                                     A_1 & 0 \\
                                    0 & A_{2}\\
                                  \end{array}
                                \right)
\mbox{ on } H=M_{1}\oplus M_{2}$$  and $$ B=\left(
                                  \begin{array}{cc}
                                     B_1 & 0 \\
                                    0 & B_{2}\\
                                  \end{array}
                                \right)
\mbox{ on }H=N_{1}\oplus N_{2}.$$ Recall from \cite{ER} that $\sigma(\Delta_{A,B})=\sigma(A)\sigma(B)-\{1\}$.
We consider the following cases.

\textbf{Case 1:} If $\lambda\in\mathbb{C}\backslash\sigma(\Delta_{A,B}),$ the result is immediate.

\textbf{Case 2:} If $\lambda\in iso\sigma(\Delta_{A,B})$ and $\lambda\neq-1$, then there exists finite sequences
$\{\mu_{i}\}_{i=1}^{n}$ and $\{\nu_{i}\}_{i=1}^{n}$, where
$\mu_{i}\in iso\sigma(A)$ and $\nu_{i}\in iso\sigma(B)$ such that $$\lambda=\mu_{i}\nu_{i}-1, \mbox{ for all }1\leq i\leq n.$$
$\lambda\not\in\sigma(\Delta_{A_{i},B_{j}})$ for all $1\leq i,j\leq 2$ other than $i=j=1$. Let $X\in\ker(\Delta_{A,B}-\lambda I)$ such that $$X:N_{1}\oplus N_{2}\longrightarrow M_{1}\oplus M_{2} \mbox{ have the representation } X=[X_{kl}]_{k,l=1}^{2}.$$
Hence $$(\Delta_{A,B}-\lambda I)(X)=\left(
                                  \begin{array}{cc}
                                    (\Delta_{A_{1},B_{1}}-\lambda I)(X_{11}) & (\Delta_{A_{1},B_{2}}-\lambda I)(X_{12}) \\
                                    (\Delta_{A_{2},B_{1}}-\lambda I)(X_{21}) & (\Delta_{A_{2},B_{2}}-\lambda I)(X_{22})\\
                                  \end{array}
                                \right)=0
.$$ Observe that $\Delta_{A_{i},B_{j}}-\lambda I$ is invertible for all $1\leq i,j\leq2$ other than $i=j=1$. Hence $X_{22}=X_{21}=X_{12}=0.$ Since $A_{1}$ and $B_{1}$ are normal, it follows from Fuglede-Putnam theorem and \cite[Theorem 2]{D2} that $$\frac{1}{1+\overline{\lambda}}A_{1}^{*}X_{11}B_{1}^{*}-X_{11}=0,$$ consequently $$X=X_{11}\oplus 0\in\ker(\Delta_{A^{*},B^{*}}-\overline{\lambda} I).$$

If $\lambda=-1$, we consider the case $-1 \in iso\sigma(\Delta_{A,B})$, that is $0\in iso\sigma(L_{A}R_{B})$, hence either $0\in iso\sigma(A)$ and $0\in iso\sigma(B)$ or $0\in iso\sigma(A)$ and  $0\not\in\sigma(B)$ or $0\in iso\sigma(B)$ and  $0\not\in\sigma(A)$. If $0\in iso\sigma(A)$ and $0\in iso\sigma(B)$. Let $X:N_{1}\oplus N_{2}\longrightarrow M_{1}\oplus M_{2}$, have the matrix representation  $X=[X_{kl}]_{k,l=1}^{2}$. If $X\in\ker(L_{A}R_{B}),$ then $\left(
                                  \begin{array}{cc}
                                     A_{1}X_{11}B_{1} & A_{1}X_{12}B_{2} \\
                                    A_{2}X_{21}B_{1} & A_{2}X_{22}B_{2}\\
                                  \end{array}
                                \right)=0
$, it follows that $X_{22}=X_{21}=X_{12}=0,$ and $ A_{1}X_{11}B_{1}=0$, hence $X_{11}\in\ker L_{A_{1}^{*}}R_{B_{1}^{*}}$. Thus $X\in\ker(L_{A^{*}}R_{B^{*}}).$ The proofs of the other remaining cases are similar.

\textbf{Case 3:} If $\lambda\in acc\sigma(\Delta_{A,B})=(acc\sigma(A)\sigma(B)-1)\cup (\sigma(A)acc\sigma(B) - 1) ,$ and $\lambda=-1$ then $0\in\sigma(A)acc\sigma(B)$ or $0\in acc\sigma(A)\sigma(B)$. Since $A$ and $B$ are polaroid, then $0\in acc\sigma(A)=\sigma_{D}(A)$ and $0\in acc\sigma(B)=\sigma_{D}(B)$.
\begin{itemize}
  \item If $0\in\sigma(A_1)$ and $0\in\sigma_{D}(B_1)$.
  Observe that $L_{A_{i}}R_{B_{j}}$ is invertible for all $1\leq i,j\leq2$ other than $i=j=1$, we have $X\in\ker(L_{A}R_{B})$. Hence $X_{12}=X_{21}=X_{22}=0.$ Since $A_{1}$ and $B_{1}$ are normal, it follows from Fuglede-Putnam theorem that $$A_{1}^{*}X_{11}B_{1}^{*}=0,$$ consequently $$X=X_{11}\oplus 0\in\ker(L_{A^{*}}R_{B^{*}}).$$

  \item If $0\in\sigma(A_1)$ and $0\in\sigma(B_2)$. Let $X\in\ker(L_{A}R_{B})$, then $X_{22}=X_{21}=X_{11}=0$. We have
  $A_1X_{12}B_2=0$ this implies that $A_1X_{12}=0$. Since $A_1$ is normal, then $A^{*}_1X_{12}=0$, consequently $A^{*}_1X_{12}B^{*}_2=0$. Hence $$X=\left(
                                  \begin{array}{cc}
                                    0 & X_{12} \\
                                   0 & 0\\
                                  \end{array}
                                \right)\in\ker(L_{A^{*}}R_{B^{*}}).$$
  \item If $0\in\sigma(A_2)$ and $0\in\sigma_{D}(B_1)$, this case will be proved similarly as the precedent one.
  \item If $0\in\sigma(A_2)$ and $0\in\sigma(B_2)$. Let $X\in\ker(L_{A}R_{B})$, then $X_{12}=X_{21}=X_{11}=0$. Since $A_{2}$ and $B^{*}_{2}$ are injective, then $X_{22}=0$. Hence $$X=0\in\ker(L_{A^{*}}R_{B^{*}}).$$
\end{itemize}
The other cases will be proved similarly.\\

\textbf{Case 4:}
If $\lambda\in acc\sigma(\Delta_{A,B})=(acc\sigma(A)\sigma(B)-1)\cup (\sigma(A)acc\sigma(B) - 1)$ and $\lambda\neq-1,$ then there exists $\mu\in\sigma(A)$ and $\nu\in\sigma(B)$ such that $\lambda=\mu\nu\in(\sigma(A)acc\sigma(B)-1)$ or $\lambda=\mu\nu\in(acc\sigma(A)\sigma(B)-1)$.
\begin{itemize}
  \item If $\mu\in\sigma(A_1)$ and $\nu\in\sigma_{D}(B_1)$.
  Observe that $\Delta_{A_{i},B_{j}}-\lambda I$ is invertible for all $1\leq i,j\leq2$ other than $i=j=1$, we have $X\in\ker(\Delta_{A,B}-\lambda I)$. Hence $X_{12}=X_{21}=X_{22}=0.$ Since $A_{1}$ and $B_{1}$ are normal, it follows from Fuglede-Putnam theorem and \cite[Theorem 2]{D2} that $$\frac{1}{1+\overline{\lambda}}A_{1}^{*}X_{11}B_{1}^{*}-X_{11}=0,$$ consequently $$X=X_{11}\oplus 0\in\ker(\Delta_{A^{*},B^{*}}-\overline{\lambda} I).$$

  \item If $\mu\in\sigma(A_1)$ and $\nu\in\sigma(B_2)$. Let $X\in\ker(\Delta_{A,B}-\lambda I)$, then $X_{22}=X_{21}=X_{11}=0$.
  Since $\mu\in\sigma(A_1)=cl\sigma_{p}(T_{1})$, then there exist a sequence $(\mu_{n})\in\sigma_{p}(A_{1})$ such that $\mu_{n}\longrightarrow\mu$, so for all $n\in\mathbb{N}$ there exist a non null vector $x_{n}\in\ker(A_1-\mu_{n})=\ker(A^{*}_1-\overline{\mu_{n}})$. We have
  $$B_{2}^{*}X_{12}^{*}(A_{1}^{*}-\overline{\mu})-(1+\overline{\lambda}-\overline{\mu}B^{*}_{2})X_{12}^{*}=0,$$ consequently $$B_{2}^{*}X_{12}^{*}(A_{1}^{*}-\overline{\mu})x_n-(1+\overline{\lambda}-\overline{\mu}B^{*}_{2})X_{12}^{*}x_n=0.$$ Since $\mu_{n}$ is an eigenvalue, so there exists a no null vector $z\in\ker(A^{*}_1-\overline{\mu_{n}})$, such that $\lim\|x_{n}-z\|=0.$ Since $B_{2}^{*}$ is injective, we obtain $X_{12}=0,$ hence $$X=0\in\ker(\Delta_{A^{*},B^{*}}-\overline{\lambda} I).$$
  \item If $\mu\in\sigma(A_2)$ and $\nu\in\sigma_{D}(B_1)$, this case will be proved similarly as the precedent one.
  \item If $\mu\in\sigma(A_2)$ and $\nu\in\sigma(B_2)$. Since $A$ has property $(\beta)$ it follows from
  \cite[Remarks 3.2]{BZZ} that $A_{2}$ has property $(\beta)$, applying \cite[Theorem 2.20]{A} we get $H_{0}(A_2-\mu)$ is closed and from \cite[Proposition 1.2.20]{LN} That $\sigma(A_2|_{H_{0}(A_2-\mu)})\subseteq\{\mu\}$. If $\sigma(A_2|_{H_{0}(A_2-\mu)})=\emptyset$, then $H_{0}(A_2-\mu)=\{0\}$, the case $\sigma(A_2|_{H_{0}(A_2-\mu)})=\{\mu\}$ is not possible, since the operator $A_2$ does not contains isolated points. Hence $H_{0}(A_2-\mu)=\{0\},$ we have $X\in\ker(\Delta_{A,B}-\lambda I)$, then $X_{21}=X_{12}=X_{11}=0$ and $A_2X_{22}B_2-(1+\lambda)X_{22}=0$, this implies that $$(A_2-\mu)X_{22}(B_2-\nu)+\nu(A_2-\mu)X_{22}+\mu X_{22}(B_2-\nu)=0.$$ if $t\in H_{0}(B_2-\nu)$, then $X_{22}t\in H_{0}(A_2-\mu)=\{0\}.$ Hence $X_{22}t=0.$ Since $t\in H_{0}(B_2-\nu)$, using properties of quasinilpotent part, we get $(B_2-\nu)(t)\in H_{0}(B_2-\nu),$ consequently $N_{2}=\overline{H_{0}(B_2-\nu)}.$ So $X_{22}=0$, hence $$X= 0\in\ker(\Delta_{A^{*},B^{*}}-\overline{\lambda} I).$$
\end{itemize}
The other cases will be proved similarly.
\end{proof}

\begin{theorem}
Suppose that $A,B^{*}\in L(H)$ are reduced by each of its eigenspaces, polaroid and have property $(\beta)$, then
$\mathcal{R}(d_{A,B}-\lambda I)$ is orthogonal to $\ker(d_{A,B}-\lambda I),$ for all $\lambda\in\mathbb{C}.$
\end{theorem}
\begin{proof} Follows from \cite[Lemma 4]{D}
\end {proof}
\begin{corollary}\cite[Lemma 2.1]{Du8}
Suppose that $A,B^{*}\in L(H)$ are p-hyponormal or log-hyponormal, then
\begin{center}
$\ker(d_{A,B}-\lambda I)\subseteq\ker(d_{A^{*},B^{*}}-\overline{\lambda} I),$ $\forall\lambda\in\mathbb{C}$\end{center}

\end{corollary}
\begin{corollary}\cite[Lemma 2.4]{CH}
Let $A, B^{*}\in L(H)$ be w-hyponormal operators such that $\ker
A\subseteq\ker A^{*}$ and $\ker B^{*}\subseteq\ker B$, then
\begin{equation*}
\ker d_{A,B}\subseteq\ker d_{A^{*},B^{*}}.
\end{equation*}
\end{corollary}
\begin{corollary}\cite[Theorem 3.6]{BL},\cite[Lemma 2.4]{DKK}
Let $A, B^{*}\in L(H)$ be w-hyponormal operators such that $\ker
A\subseteq\ker A^{*}$ and $\ker B^{*}\subseteq \ker B$, then
\begin{equation*}
\ker \delta_{A,B}\subseteq \ker \delta_{A^{*},B^{*}}.
\end{equation*}
\end{corollary}
\begin{corollary}\cite[Theorem 2.5]{DKK}
Let $A, B^{*}\in L(H)$. If $A\in\mathcal{A}(s_1,t_1)$  and $B^{*}\in\mathcal{A}(s_2,t_2)$, $0<s_1,s_2,t_1,t_2\leq 1$ are such that
$\ker A\subseteq\ker A^{*}$ and $\ker B^{*}\subseteq\ker B$, then
\begin{equation*}
\ker \delta_{A,B}\subseteq \ker \delta_{A^{*},B^{*}}.
\end{equation*}

\end{corollary}

As a nice application of our main results the Fuglede Putnam theorem for k-quasi-class $\mathcal{A}$ operators which contains all the precedent classes of operators.

\begin{theorem}
Let $A, B^{*}\in LH)$ be k-quasi-class A operators, then
$$\ker(d_{A,B}-\lambda I)\subseteq\ker(d_{A^{*},B^{*}}-\overline{\lambda} I),$$ for all non null complex number $\lambda$.
\end{theorem}
\begin{proof}
We know from \cite[Theorem 2.4]{GF1} that k-quasi-class $\mathcal{A}$ operators are polaroid and from \cite[Lemma 11]{TJKU} that k-quasi-class $\mathcal{A}$ operators have property $(\beta)$. Since by \cite[Lemma 13]{TJKU}, $A$ and $B^{*}$ are reduced by each of its eigenspaces, then the conclusion follows from Theorem \ref{theorem1} and Theorem \ref{theorem2}.
\end{proof}
\begin{theorem}
Let $A, B^{*}\in L(H)$ be k-quasi-class $\mathcal{A}$ operators such that $\ker
A\subseteq\ker A^{*}$ and $\ker B^{*}\subseteq\ker B$, then
\begin{equation*}
\ker d_{A,B}\subseteq\ker d_{A^{*},B^{*}}.
\end{equation*}
\end{theorem}
\begin{proof}
 The conditions $\ker A\subseteq\ker A^{*}$ and $\ker B^{*}\subseteq\ker B$ implies that $0$ is normal eigenvalue of both $A$ and $B^{*}.$
it follows from Theorem \ref{theorem1} and Theorem \ref{theorem2} that $\ker d_{A,B}\subseteq\ker d_{A^{*},B^{*}}.$
\end{proof}

\end{document}